\documentclass[12pt,letterpaper, leqno]{amsart}

\usepackage{amsmath}
\usepackage{amsfonts}
\usepackage{amssymb}
\usepackage{amsthm}
\usepackage{fullpage}

\newcommand{\CC}{\mathbb{C}}
\newcommand{\FF}{\mathbb{F}}

\newcommand{\QQ}{\mathbb{Q}}
\newcommand{\ZZ}{\mathbb{Z}}
\newcommand{\R}{\mathbb{R}}
\newcommand{\ndiv}{\nmid}

\newcommand{\frakp}{\mathfrak{p}}
\numberwithin{equation}{section}
\newtheorem{theorem}[equation]{Theorem}
\newtheorem{cor}[equation]{Corollary}
\newtheorem{lemma}[equation]{Lemma}
\newtheorem{conjecture}[equation]{Conjecture}
\theoremstyle{definition}
\newtheorem{remark}[equation]{Remark}

\DeclareMathOperator{\Conj}{Conj}
\DeclareMathOperator{\Frob}{Frob}

\DeclareMathOperator{\GL}{GL}

\DeclareMathOperator{\Li}{Li}
\DeclareMathOperator{\Norm}{Norm}
\DeclareMathOperator{\rank}{rank}
\DeclareMathOperator{\Real}{Re}

\DeclareMathOperator{\SO}{SO}
\DeclareMathOperator{\ST}{ST}

\DeclareMathOperator{\SU}{SU}
\DeclareMathOperator{\Sym}{Sym}
\DeclareMathOperator{\Trace}{Trace}

\title{An application of the effective Sato-Tate conjecture}
\date{June 5, 2015}

\author{Alina Bucur}
\address{Department of Mathematics, University of California at San Diego, 9500 Gilman Dr \#0112, La Jolla, CA 92093}
\email{alina@math.ucsd.edu}
\thanks{Bucur was supported by the Simons Foundation (collaboration grant \#244988) and UCSD
(Hellmann fellowship)}

\author{Kiran S. Kedlaya}
\address{Department of Mathematics, University of California at San Diego, 9500 Gilman Dr \#0112, La Jolla, CA 92093}
\email{kedlaya@ucsd.edu}
\urladdr{http://kskedlaya.org}
\thanks{Kedlaya was supported by NSF (grant DMS-1101343) and
UCSD (Stefan E. Warschawski professorship)}

\begin{document}

\begin{abstract}
Based on the Lagarias-Odlyzko effectivization of the Chebotarev density theorem,
Kumar Murty gave an effective version of the Sato-Tate conjecture for an elliptic curve conditional on analytic continuation and Riemann hypothesis for the symmetric power $L$-functions.
We use Murty's analysis to give a similar conditional effectivization of the generalized Sato-Tate conjecture for an arbitrary motive. As an application, we give a conditional upper bound of the form $O((\log N)^2 (\log \log 2N)^2)$
for the smallest prime at which two given rational elliptic curves with conductor at most $N$ have Frobenius traces of opposite sign.
\end{abstract}

\subjclass[2010]{11G05, 11R44}

\maketitle

Let $\pi(x)$ denote the number of prime numbers less than or equal to $x$.
Hadamard and de la Vall\'ee-Poussin proved the prime number theorem
\[
\pi(x) = (1 + o(1)) \frac{x}{\log x}
\]
by exploiting the relationship
between prime numbers and the zeroes of the Riemann zeta function. Assuming Riemann's
hypothesis that the zeroes in the critical strip $0 \leq \Real(s) \leq 1$ all lie on the line
$\Real(s) = 1/2$, one gets a much more precise estimate for $\pi(x)$:
\[
\pi(x) = \Li(x) + O(x^{1/2} \log x) \qquad \left(\Li(x) = \int_2^x \frac{dt}{\log t}\right).
\]
Here and throughout this paper, the 
implied constant in the big-O notation is absolute and effectively computable, and the assertion
applies for all $x$ exceeding some other effectively computable absolute constant.
(For details, see any introductory text on analytic number theory, e.g., \cite{davenport}.)

A similar paradigm applies to the distribution of values of various other functions of prime numbers,
or more generally of prime ideals in a number field $K$: one gets an asymptotic result using some limited
analytic information about $L$-functions, but under the analogue of the Riemann hypothesis one gets
an estimate with a small effective error term. For example, 
for the Chebotarev density theorem (describing the distribution of Frobenius classes for a fixed Galois extension of $K$), this  effectivization process was described by Lagarias and Odlyzko \cite{lagarias-odlyzko}.
More recently, the Sato-Tate conjecture (describing the distribution
of Frobenius traces for a fixed elliptic curve over $K$) 
has been established for $K$ totally real through the efforts of Taylor et al.\ (see \cite{BLGG11} for a definitive
result); in this case, the effectivization process had
been described previously by Kumar Murty \cite{kumar}.

The previous two examples can both be subsumed into a generalized Sato-Tate conjecture for an arbitrary motive,
taking an Artin motive in the case of Chebotarev and the 1-motive of an elliptic curve in the case of Sato-Tate.
The first purpose of this paper is to explain, under suitable analytic hypotheses
on motivic $L$-functions (Conjecture~\ref{conj:motivic L-functions}),
how to obtain effective error bounds for the generalized Sato-Tate conjecture. The technique is essentially an application of Weyl-type explicit formulas as in \cite{lagarias-odlyzko} and \cite{kumar}; in fact, Murty's treatment of the analytic arguments in \cite{kumar} is already general enough to apply to arbitrary motives,
so it is not necessary to redo any of the complex analysis.
(Murty was practically forced to work at this level of generality to handle the usual Sato-Tate conjecture, because he needed his arguments to apply uniformly over symmetric powers. Here we apply them uniformly over representations of a compact Lie group.)

The second purpose of this paper is to indicate an application of the effective form of the generalized
Sato-Tate conjecture to a classical question about the arithmetic of elliptic curves.
Let $E_1$ and $E_2$ be nonisogenous elliptic curves over $K$, neither having complex multiplication.
The isogeny theorem of Faltings \cite{faltings} implies that there exists a prime ideal $\frakp$ of $K$
at which $E_1, E_2$ both have good reduction and have distinct Frobenius traces. 
In particular, for any fixed prime $\ell$, there exists a prime ideal $\frakp$ of $K$ at which the Frobenius traces of $E_1, E_2$ differ modulo $\ell$.
Assuming the generalized Riemann hypothesis for Artin $L$-functions,
one can use the effective form of the Chebotarev density theorem
(as suggested by Serre in \cite{serre-chebotarev}; see also Corollary~\ref{C:Serre bound})
to show the least norm of such a prime ideal is
\[
O((\log N)^2 (\log \log 2N)^b)
\]
for some fixed $b\geq0$. 
Assuming the generalized Riemann hypothesis for $L$-functions of the form $L(s,\Sym^m E_1 \otimes \Sym^n E_2)$,
we use the effective form of the generalized Sato-Tate conjecture for the abelian surface $E_1 \times_K E_2$
to obtain a similar bound for the least norm of a prime ideal at which the Frobenius traces of $E_1, E_2$
have opposite sign (Theorem~\ref{T:two elliptic curves}).
In both cases, the optimal bound is most likely closer to $O(\log N)$, but by analogy with the problem of finding the least quadratic nonresidue modulo $N$, it is unlikely that one can do better than $O((\log N)^2)$  using
$L$-function methods.

Although we will not do so here, we mention that the framework of the generalized Sato-Tate conjecture includes many additional questions about distinguishing $L$-functions, a number of which have been considered previously.
For instance, Goldfeld and Hoffstein \cite{goldfeld-hoffstein} established an upper bound on the first distinguishing coefficient for a pair of holomorphic Hecke newforms, by an argument similar to ours but with a milder analytic hypothesis (the Riemann hypothesis for the Rankin-Selberg convolutions of the two forms with themselves and each other).
Sengupta \cite{sengupta} carried out the analogous analysis with the Fourier coefficients
replaced by normalized Hecke eigenvalues (this only makes a difference when the weights are distinct). The analogue of Serre's argument for modular forms was given by Ram Murty
\cite{ram} and subsequently extended to Siegel modular forms by Ghitza
\cite{ghitza} for Fourier coefficients and Ghitza and Sayer \cite{ghitza-sayer} for Hecke eigenvalues.

\section{Motivic $L$-functions and motivic Galois groups}

We begin by recalling the conjectural properties of motivic $L$-functions, as in \cite{serre-motives}.

Fix two number fields $K,L$.
Let $\mathcal M$ be a pure motive of weight $w$ over $K$ with coefficients in $L$.
For each prime ideal $\frakp$ of $K$, 
let $G_{\frakp}$ be a decomposition subgroup of $\frakp$ inside the absolute Galois group $G_K$,
let $I_\frakp$ be the inertia subgroup of $G_{\frakp}$,
and let $\Frob_{\frakp} \in G_{\frakp}/I_{\frakp}$ be the Frobenius element.
The \emph{Euler factor} of $\mathcal M$ at $\frakp$ (for the automorphic normalization) is the function
\[
L_{\frakp}(s,\mathcal M) =  \det(1 - \Norm(\frakp)^{-s-w/2} \Frob_{\frakp}, V_{v}(\mathcal M)^{I_\frakp} \otimes_{L_v} \CC)^{-1}
\]
for $v$ a finite place of $L$ equipped with an embedding $L_v \hookrightarrow \CC$
and $V_{v}(\mathcal M)$ the $v$-adic \'etale realization of $\mathcal M$
equipped with its action of $G_{\frakp}$.
It is clear that this definition does not depend on the choice
of $G_{\frakp}$; it is conjectured also not to depend on $v$ or the embedding $L_v \hookrightarrow \CC$, 
and this is known when $\mathcal M$ has good reduction
at $\frakp$ (which excludes only finitely many primes).

The ordinary $L$-function of $\mathcal M$ is the the Euler product
\[
L(s,\mathcal M) = \prod_{\frakp} L_{\frakp}(s,\mathcal M).
\]
For each infinite place $\infty$ of $K$, there is also an archimedean Euler factor defined as follows.
Put
\[
\Gamma_{\R}(s) = \pi^{-s/2} \Gamma(s/2), \qquad
\Gamma_{\CC}(s) = 2^{-s} \pi^{-s} \Gamma(s).
\]
Form the Betti realization of $\mathcal M$ at $\infty$ and the spaces $H^{p,q}$ for $p+q=w$,
and put $h^{p,q} = \dim H^{p,q}$. Note that complex conjugation takes $H^{p,q}$ to $H^{q,p}$
and thus acts on $H^{w/2,w/2}$; let $h^+$ and $h^-$ be the dimensions of the positive and negative eigenspaces
(both taken to be $0$ if $w$ is odd). Then put
\[
L_{\infty}(s,\mathcal M) = 
\Gamma_{\R}(s)^{h^+} \Gamma_{\R}(s+1)^{h^-}
\prod_{p+q=w,p<q} \Gamma_{\CC}(s+w/2-p)^{h^{p,q}}.
\]
The completed $L$-function is then defined as
\[
\Lambda(s,\mathcal M) = N^{s/2} L(s,\mathcal M) \prod_{\infty} L_{\infty}(s,\mathcal M),
\]
for $N$ the absolute conductor of $\mathcal M$ (i.e., the norm from $K$ to $\QQ$ of the conductor ideal of $\mathcal M$).

\begin{conjecture} \label{conj:motivic L-functions}
Let $d$ be the dimension of the fixed subspace of 
the motivic Galois group of $\mathcal M(-w/2)$ (taken to be $0$ if $w$ is odd).
\begin{enumerate}
\item[(a)]
The function $s^d (1-s)^d \Lambda(s,\mathcal M)$ (which is defined \emph{a priori} for $\Real(s) > 1$)
extends to an entire function on $\CC$ of order $1$
which does not vanish at $s=0,1$. (Recall that an entire function $f: \CC \to \CC$ is of \emph{order $1$}
if $f(z) e^{-\mu|z|}$ is bounded for each $\mu > 1$.)
\item[(b)]
Let $\mathcal M^*$ denote the Cartier dual of $\mathcal M$. Then there exists $\epsilon \in \CC$ with $|\epsilon| = 1$ such that
$\Lambda(1-s,\mathcal M) = \epsilon \Lambda(s,\mathcal M^*)$ for all $s \in \CC$.
\item[(c)]
The zeroes of $\Lambda(s,\mathcal M)$ all lie on the line $\Real(s) = 1/2$.
\end{enumerate}
\end{conjecture}

\begin{remark}
At present, the most promising approach to proving parts (a) and (b) of Conjecture~\ref{conj:motivic L-functions} 
for a given $\mathcal M$ is to show that $\Lambda(s,\mathcal M)$ coincides with a potentially automorphic $L$-function. 
For example, this is known for the symmetric power $L$-functions of an elliptic curve over a totally real number
field \cite{BLGG11} and for the Rankin-Selberg product of two such $L$-functions \cite{harris}.
This implies (a) and (b) for such $L$-functions, using the work of Gelbart and Shahidi
\cite{gelbart-shahidi} to verify the order 1 condition.
See \cite{murty-murty} for an overview of how to use potential automorphy to deduce the Sato-Tate conjecture.
Part (c), the analogue of the Riemann hypothesis, is unknown in all cases.
\end{remark}

\section{Equidistribution and motivic $L$-functions}

We next recall how to use the analytic information about motivic $L$-functions provided by 
Conjecture~\ref{conj:motivic L-functions}
to obtain equidistribution statements with small effective error bounds. This combines the general approach to equidistribution described in 
\cite[Appendix to Chapter~1]{serre-abelian} with the extraction of effective bounds from $L$-functions
described in \cite{lagarias-odlyzko, kumar}.

Take $\mathcal M$ as before. 
The \emph{motivic Sato-Tate group} of $\mathcal M$ is the kernel of the map from the motivic Galois group of
$\mathcal M \oplus L(1)$ to the motivic Galois group of the Tate motive $L(1)$; this is a subgroup of the usual motivic
Galois group of $\mathcal M$. Taking a compact form of the motivic Sato-Tate group
yields the \emph{Sato-Tate group} $G$. From the construction, one obtains a sequence $\{g_{\frakp}\}$
in the space $\Conj(G)$ of conjugacy classes of $G$ corresponding to prime ideals of good reduction,
such that for any motive $\mathcal N$ pure of weight $k$ in the Tannakian
category generated by $\mathcal M$, the characteristic polynomial of $\Norm(\frakp)^{-k/2} \Frob_{\frakp}$
on any \'etale realization of $\mathcal N$ equals the characteristic polynomial of $g_{\frakp}$ on
the corresponding representation of $G$.

Topologize $\Conj(G)$ as a quotient of $G$, and equip it with the measure $\mu$ with the property that for
any continuous function $F: \Conj(G) \to \CC$, $\mu(F)$ is the Haar measure of the pullback of $F$ to $G$.
The statement that the $g_{\frakp}$ are equidistributed in $\Conj(G)$ would mean that for any $F$,
if we write $F(\frakp)$ as shorthand for $F(g_{\frakp})$, then
\begin{equation} \label{eq:equidistribution}
\sum_{\Norm(\frakp) \leq x} F(\frakp) = ( \mu(F) + o(1)) \sum_{\Norm(\frakp) \leq x} 1.
\end{equation}
By the Peter-Weyl theorem, we have
\begin{equation} \label{eq:peter-weyl}
F = \sum_\chi \mu(F \overline{\chi}) \chi
\end{equation}
where the sum runs over irreducible characters $\chi$ of $G$,
so it suffices to check \eqref{eq:equidistribution} for these characters. For such a character $\chi$,
let $L(s,\chi)$ be the $L$-function of the motive  corresponding to $\chi$ in the Tannakian category generated by $\mathcal M$. One then shows that if\footnote{In fact, somewhat less analytic information is needed;
one only needs $L(s,\chi)$ to extend to a meromorphic function on $\Real(s) \geq 1$ with no zeroes
or poles except for a simple pole at $s=1$ in case $\chi$ is trivial. This resembles the standard proof of the
prime number theorem; see \cite[Theorem~1, Appendix to Chapter~1]{serre-abelian}.}
parts (a) and (b) of Conjecture~\ref{conj:motivic L-functions}
hold for each $L(s,\chi)$, then \eqref{eq:equidistribution} holds. 

Assume now that Conjecture~\ref{conj:motivic L-functions}, including part (c), holds for each $L(s,\chi)$. 
One can obtain information about the average behavior of $\chi(\frakp)$ by computing a suitable contour integral of the logarithmic derivative of $L(s,\chi)$, as in \cite{lagarias-odlyzko}. By keeping careful track of
the dependence on various factors, as in \cite[Proposition~4.1]{kumar}, one obtains the following estimate: for $d_{\chi} = \dim(\chi) [K:\QQ]$, 
\begin{equation} \label{eq:explicit formula1}
\sum_{\Norm(\frakp) \leq x} \chi(\frakp) \log \Norm(\frakp) = \mu(\chi) x + O(d_\chi x^{1/2} \log x  \log (N (x + d_\chi))).
\end{equation}
Using Abel partial summation, it then follows that
\begin{equation} \label{eq:explicit formula2}
\sum_{\Norm(\frakp) \leq x} \chi(\frakp)  = \mu(\chi) \Li(x) + O(d_\chi x^{1/2} \log (N (x + d_\chi))).
\end{equation}
Since the implied constant in the big-O notation is absolute (and in particular independent of $G$)
and effectively computable, one can obtain an effective bound on 
\[
\sum_{\Norm(\frakp) \leq x} F(\frakp) - \mu(F) \Li(x)
\]
for a general continuous function $F: \Conj(G) \to \CC$
by summing \eqref{eq:explicit formula2} over the terms of the expansion \eqref{eq:peter-weyl}.
In practice, one gets a slightly better result by applying \eqref{eq:explicit formula2} only to characters
of small dimension, lumping the large characters directly into the error term.
Even with this refinement, though, to obtain reasonable results one must impose enough regularity on $f$ to get sufficient decay for the coefficients in \eqref{eq:peter-weyl}.
This was demonstrated explicitly for the case of $\SU(2)$ by Murty \cite{kumar}, whose arguments we recall in
the next section.

\section{The case of an elliptic curve}
\label{sec:elliptic curves}

In this section, we recall the treatment of the effective Sato-Tate conjecture for elliptic curves by Murty \cite{kumar}
and indicate how it arises as a specialization of the preceding discussion. Our exposition is somewhat complementary to that of \cite{kumar}, where the explicit formula \eqref{eq:explicit formula1} and the application to
the Lang-Trotter conjecture are treated in detail; we instead take \eqref{eq:explicit formula1} as a black
box and discuss how effective Sato-Tate emerges from it in detail. (Concretely, this means that we
apply Lemma~\ref{lemma:vinogradov} with slightly different parameters than in \cite{kumar}.)

For $E$ an elliptic curve over a number field $K$ and $\frakp$ a prime ideal of $K$ at which $E$ has good reduction,
let $a_{\frakp} = a_{\frakp}(E)$ be the
Frobenius trace of $E$ at $\frakp$, so that $\Norm(\frakp) + 1 - a_{\frakp}$ is the number of rational points on the reduction of $E$ modulo $\frakp$. 
Then define the Frobenius angle $\theta_{\frakp} = \theta_{\frakp}(E) \in [0, \pi]$ by the formula
\[
1 - a_{\frakp}(E)T + \Norm(\frakp) T^2 = (1 - \Norm(\frakp)^{1/2} e^{i \theta} T)
(1 - \Norm(\frakp)^{1/2} e^{-i \theta}T).
\]
Let $\mu_{\ST}$ denote the Sato-Tate measure, so that 
\[
\mu_{\ST}(f) = \int_0^{\pi} \frac{2}{\pi} \sin^2 \theta f(\theta) \,d\theta.
\]
For $I$ an interval, let $\chi_I$ denote the characteristic function. We prove the following theorem.
(As usual, the implied constant in the big-O notation is absolute and effectively computable.)
\begin{theorem}[after Murty] \label{T:effective Sato-Tate}
Let $E$ be an elliptic curve over a number field $K$ without complex multiplication.
Let $N$ denote the absolute conductor of $E$.
Assume that $L(s,\Sym^k E)$ satisfies Conjecture~\ref{conj:motivic L-functions} for all $k \geq 0$.
Then for any closed subinterval $I$ of $[0, \pi]$,
\[
\sum_{\Norm(\frakp) \leq x, \frakp \ndiv N} \chi_I(\theta_\frakp)
= \mu_{\ST}(I) \Li(x) + O([K:\QQ]^{1/2} x^{3/4} (\log (Nx))^{1/2} ).
\]
\end{theorem}

The weaker statement that $\sum_{\Norm(\frakp) \leq x, \frakp \ndiv N} \chi_I(\theta_\frakp)
\sim \mu_{\ST}(I) \Li(x)$ is the Sato-Tate conjecture, which is known unconditionally when $K$ is totally real.
See \cite{murty-murty} for more discussion.

To prove Theorem~\ref{T:effective Sato-Tate}, we first note that the number of primes dividing $N$ (which includes all primes of bad reduction) is $O(\log N)$, which is subsumed by our error term. We can thus safely
neglect bad primes in what follows.

We next introduce a family of functions $F$ for which we have control over the coefficients appearing in \eqref{eq:peter-weyl}, which we will use to approximate the characteristic function $\chi_I$.
Since $E$ has no complex multiplication, its Sato-Tate group is $\SU(2)$,
whose characters are
\begin{equation} \label{eq:SU2 characters}
\chi_k(\theta) = \sum_{j=0}^k e^{(k-2j)i \theta} \qquad (k=0,1,\dots).
\end{equation}
Thus expanding $F$ in terms of the $\chi_k$ amounts to ordinary Fourier analysis: 
if we formally extend $F$ to an even function on $[-\pi, \pi]$
and form the ordinary Fourier decomposition
\begin{equation} \label{eq:F expansion1}
F(\theta) = c_0 + \sum_{k=1}^\infty 2c_k \cos (k \theta),
\end{equation}
we can then write
\begin{equation} \label{eq:F expansion2}
F(\theta) = \sum_{k=0}^\infty (c_k - c_{k+2}) \chi_k(\theta).
\end{equation}
We can thus avail ourselves of a construction of Vinogradov \cite[Lemma~12]{vinogradov}.
For now, we take the construction as a black box; we will recall the method of proof in the context of the generalized Sato-Tate conjecture in \S \ref{sec:general case}.

\begin{lemma}\label{lemma:vinogradov}
Let $r$ be a positive integer, and let $A, B, \Delta$ be real numbers satisfying
\[0 <\Delta <\frac{1}{2}, \qquad \Delta \leq B-A \leq 1-\Delta. \] 
Then there exists a continuous periodic function $D_{A,B}= D:\R \to \R$ with period $1$ that satisfies the following conditions.
\begin{enumerate}
\item For $A+\frac{1}{2}\Delta \leq x \leq B -\frac{1}{2}\Delta$, $D(x) = 1$. 
\item For $B+\frac{1}{2}\Delta \leq x \leq 1+A -\frac{1}{2}\Delta$,  $D(x) = 0$.
\item For $x$ in the remainder of the interval $\left [A -\frac{1}{2}\Delta, 1+ A -\frac{1}{2}\Delta \right]$, $0 \leq D(x) \leq 1$.
\item $D(x)$ has a Fourier series expansion of the form \[D(x) = \sum_{m\geq 0} \left( a_m \cos(2m \pi x) + b_m \sin(2m\pi x) \right)\] in which $a_0 = B-A$ and for all $m\geq 1,$  
\begin{equation}\label{eq:vinogradov}
|a_m|, |b_m| \leq \min \left\{ 2(B-A), \frac{2}{\pi m},
\frac{2}{\pi m} \left(\frac{r}{\pi m \Delta}\right)^r \right\}.
\end{equation}

\end{enumerate}
\end{lemma}

Leaving the choices of $A,B,\Delta,r$ unspecified for the moment,
let us define $D$ as in Lemma~\ref{lemma:vinogradov},
then define the function $F_{A,B}:\R\to \R$ by $F_{A,B}(\theta) = \displaystyle D\left( \frac{\theta}{2\pi}\right) + D\left(-\frac{\theta}{2\pi}\right).$ 

Its Fourier series has the form 
\begin{equation} \label{eq:Fourier series}
F_{A,B}(\theta) = \sum_{m \in \ZZ} c_{m,A,B} e^{im\theta},
\end{equation}
where \[c_{0,A,B} = 2a_0 = 2(B-A)\] and  \[c_{m,A,B}= c_{-m,A,B} = a_m \textrm{ for all }m \geq 1.\]

Let $M$ be a positive integer (to be specified later).
By truncating the Fourier series \eqref{eq:Fourier series} and using the bounds for the Fourier coefficients \eqref{eq:vinogradov}, we see that 
\[F_{A,B}(\theta) = \sum_{|m|\leq M} c_{m,A,B} e^{im\theta} + O\left( \frac{(r/\pi)^r}{M^r\Delta^{r}} \right) .\]
Rewriting in terms of the characters of $\SU(2)$ gives
\begin{equation}\label{eq:Fchitrunc}
F_{A,B}(\theta) = (c_{0,A,B}-c_{2,A,B}) + \sum_{m=1}^{M-2} (c_{m,A,B} - c_{m+2,A,B}) \chi_m(\theta) + O\left( \frac{(r/\pi)^r}{M^r\Delta^{r+1}} \right).
\end{equation}
If we take $r = 1$, then \eqref{eq:vinogradov} implies
\[
\sum_{m=1}^{M-2} m \left| c_{m,A,B} - c_{m+2,A,B} \right| 
= O\left( \frac{\log M}{\Delta} \right).
\]

Applying \eqref{eq:explicit formula2} then yields
\begin{align} \label{eq:estimate for F sum}
\lefteqn{\sum_{\Norm(\frakp) \leq x, \frakp \ndiv N} F_{A,B}(\theta_{\frakp})} \\
&
\nonumber
= (c_{0,A,B} - c_{2,A,B}) \Li(x) + O\left( \frac{[K:\QQ] x^{1/2}  \log (N(x+M)) \log M}{\Delta} \right) + O\left( \frac{x}{M \Delta \log x}
\right).
\end{align}
To deduce Theorem~\ref{T:effective Sato-Tate},
note that on one hand, the characteristic function of the interval $I = [2\pi \alpha,2 \pi \beta]$ is bounded from above
by $F_{\alpha-\Delta/2,\beta+\Delta/2}$ and from below by $F_{\alpha+\Delta/2,\beta-\Delta/2}$.
On the other hand, the quantities $c_{0,\alpha-\Delta/2,\beta+\Delta/2} - c_{2,\alpha-\Delta/2,\beta+\Delta/2}$
and $c_{0,\alpha+\Delta/2,\beta-\Delta/2} - c_{2,\alpha+\Delta/2,\beta-\Delta/2}$ each differ from $\mu_{ST}(I)$
by $O(\Delta)$.
We obtain the theorem by balancing the error terms in \eqref{eq:explicit formula2} with each other and with $O(\Delta \Li(x))$ by setting
\[
\Delta = x^{-1/4} [K:\QQ]^{1/2} (\log x) (\log (Nx))^{1/2},
\qquad
M = \lceil \Delta^{-2} \rceil.
\]

\section{The case of two elliptic curves}
\label{sec:two elliptic curves}

We now consider a variant of the previous situation.

\begin{theorem} \label{T:two elliptic curves}
Let $E_1, E_2$ be two $\overline{\mathbb{Q}}$-nonisogenous elliptic curves over a number field $K$, neither having complex multiplication.
Let $N$ be the product of the absolute conductors of $E_1$ and $E_2$.
For each prime ideal $\frakp$ of $K$ not dividing $N$, let $\theta_{1,\frakp}, \theta_{2, \frakp}$ be
the Frobenius angles of $E_1, E_2$ at $\frakp$.
Assume that the $L$-functions $L(s,\Sym^i E_1 \otimes \Sym^j E_2)$ for $i,j=0,1,\dots$
all satisfy Conjecture~\ref{conj:motivic L-functions}.
Then for any closed subintervals $I_1, I_2$ of $[0, \pi]$,
\[
\sum_{\Norm(\frakp) \leq x, \frakp \ndiv N} \chi_{I_1}(\theta_{1,\frakp})
\chi_{I_2}(\theta_{2,\frakp}) = \mu_{\ST}(I_1) \mu_{\ST}(I_2) \Li(x) + O([K:\QQ]^{1/3} x^{5/6} (\log (Nx))^{1/3} ).
\]
\end{theorem}
To prove Theorem~\ref{T:two elliptic curves}, note that 
by \cite[\S 4.2]{fkrs}
the Sato-Tate group of the 1-motive associated to $E_1 \times_K E_2$ is $\SU(2) \times \SU(2)$, whose conjugacy classes we identify with $[0, \pi] \times [0, \pi]$. 
For real numbers $A_1, B_1, A_2, B_2, \Delta$ with
\[0 <\Delta <\frac{1}{2}, \qquad \Delta \leq B_1-A_1, B_2 - A_2 \leq 1-\Delta, \] 
put $d_{m_i,A_i,B_i} = c_{m_i,A_i,B_i} - c_{m_i+2,A_i,B_i}$; then
\[
F_{A_1,B_1}(\theta_1) F_{A_2,B_2}(\theta_2) = 
\sum_{m_1,m_2=0}^M d_{m_1,A_1,B_1} d_{m_2,A_2,B_2}
\chi_{m_1}(\theta_1) \chi_{m_2}(\theta_2)
+ O\left( \frac{(r/\pi)^r}{M^r \Delta^{2r}} \right).
\]
If we take $r=1$, then \eqref{eq:vinogradov} implies
\[
\sum_{m_1,m_2=0}^{M-2} m_1m_2  \left| d_{m_1,A_1,B_1}  d_{m_2,A_2,B_2} \right| 
= O\left( \frac{(\log M)^2}{\Delta^2} \right).
\]
We thus have
\begin{align}
\label{eq:primary estimate}
\lefteqn{\sum_{\Norm(\frakp) \leq x, \frakp \ndiv N}
F_{A_1,B_1}(\theta_{1,\frakp})F_{A_2,B_2}(\theta_{2,\frakp})}\\
\nonumber
&= d_{0,A_1,B_1} d_{0,A_2,B_2} \Li(x) + O \left( \frac{[K:\QQ] x^{1/2}  \log (N(x+M)) (\log M)^2}{\Delta^2} \right) + O \left( \frac{x}{M \Delta^2 \log x} \right).
\end{align}
As in the proof of Theorem~\ref{T:effective Sato-Tate}, all that remains is to balance the error terms
in \eqref{eq:primary estimate} with each other and with $O(\Delta \Li(x))$ by taking
\[
\Delta = x^{-1/6} [K:\QQ]^{1/3} (\log x) (\log (Nx))^{1/3},
\qquad
M = \lceil \Delta^{-3} \rceil.
\]
This yields Theorem~\ref{T:two elliptic curves} as desired.

Theorem~\ref{T:two elliptic curves} immediately implies that
there exists a prime ideal $\frakp$ of $K$
with $\Norm(\frakp) = O([K:\QQ]^2 (\log N)^2 (\log \log N)^6)$ for which $E_1$ and $E_2$ have good reduction and $a_{\frakp}(E_1) \neq a_{\frakp}(E_2)$. Namely, if we fix two disjoint intervals $[A_1, B_1]$ and $[A_2, B_2]$,
then for $\Delta$ as above, the count of prime ideals is at least
$c_1 \Li(x) (1 - c_2 \Delta)$ for some absolute constants $c_1, c_2$.
This count is forced to be positive as soon as $c_2 \Delta < 1$, proving the claim.

Note however that applying Theorem~\ref{T:two elliptic curves} is not the correct optimization for this problem,
as the parameters of the proof were tuned to optimize for large $x$ rather than for small $x$.
Changing this optimization leads to a sharper result.

\begin{theorem} \label{T:distinguish elliptic curves}
With hypotheses and notation as in Theorem~\ref{T:two elliptic curves},
there exists a prime ideal $\frakp$ not dividing $N$
with $\Norm(\frakp) = O([K:\QQ]^2 (\log N)^2 (\log \log 2N)^2)$ 
such that $a_{\frakp}(E_1)$ and $a_{\frakp}(E_2)$ are nonzero and of opposite sign.
\end{theorem}
\begin{proof}
Fix once and for all some $\epsilon > 0$,
then put 
\[
A_1 = \epsilon, \quad B_1 = 1/4 - \epsilon, \quad A_2 = 1/4 + \epsilon, \quad B_2 = 1/2 - \epsilon.
\]
We set notation as in the proof of Theorem~\ref{T:two elliptic curves} except that now we take $r = 2$.
In this case, we have
\[
\sum_{m_1,m_2=0}^{M-2} m_1m_2  \left| d_{m_1,A_1,B_1}  d_{m_2,A_2,B_2} \right| 
= O\left( \frac{1}{\Delta^4} \right)
\]
and
\begin{align*}
\lefteqn{\sum_{\Norm(\frakp) \leq x, \frakp \ndiv N}
F_{A_1,B_1}(\theta_{1,\frakp})F_{A_2,B_2}(\theta_{2,\frakp})}\\
&= d_{0,A_1,B_1} d_{0,A_2,B_2} \Li(x) + O \left( \frac{[K:\QQ]x^{1/2}  \log (N(x+M))}{\Delta^4} \right) + O \left( \frac{x}{M^2 \Delta^4 \log x} \right).
\end{align*}
This time, balancing the error terms with $O(\Delta \Li(x))$ yields
\[
\Delta = x^{-1/10} [K:\QQ]^{1/5} (\log x)^{1/5} (\log (Nx))^{1/5},
\qquad
M = \lceil \Delta^{-5/2} \rceil.
\]
The proof of Theorem~\ref{T:two elliptic curves}
implies that there exists an absolute constant $c$ such that there is a prime ideal of the desired form
whenever $c \Delta < 1$. This is the same as $x > c^5 [K:\QQ]^2 (\log x)^2 (\log (Nx))^{2}$,
from which the claim follows.
\end{proof}

\begin{remark}
In the proof of Theorem~\ref{T:distinguish elliptic curves}, it is not really necessary to take $\Delta$
decreasing to 0. It would suffice to fix $A_1, B_1, A_2, B_2, \Delta$ so that the functions
$F_{A_1,B_1}$ and $F_{A_2,B_2}$ have disjoint support, then balance the error terms in
\eqref{eq:primary estimate} against the main term.
\end{remark}

\begin{remark}
The conclusion of Theorem~\ref{T:distinguish elliptic curves}
remains true if $E_1, E_2$ are isogenous over $\overline{\mathbb{Q}}$ but not over $K$: in this case they differ by a twist, so the claim reduces 
directly to effective Chebotarev \cite[Th\'eor\`eme~5]{serre-chebotarev}.
\end{remark}

For the remainder of \S\ref{sec:two elliptic curves},
retain notation as in Theorem~\ref{T:distinguish elliptic curves} but assume for simplicity that $K = \QQ$.

\begin{remark} \label{R:compare with l-adic}
Theorem~\ref{T:distinguish elliptic curves},
which distinguishes two Frobenius traces using their archimedean
behavior, should be compared with similar results which distinguish the traces using their mod-$\ell$
behavior for some prime $\ell$. 
For example, in \cite[\S 8.3, Th\'eor\`eme~21]{serre-chebotarev}, Serre shows that 
there exists a prime number $p$ not dividing $N$
with $p = O((\log N)^2 (\log \log 2N)^{12})$ 
such that $a_{p}(E_1)$ and $a_{p}(E_2)$ differ modulo some auxiliary prime $\ell$.

Both this argument and Theorem~\ref{T:distinguish elliptic curves} give upper bounds
on the norm of a prime ideal $\frakp$ for which  $a_{\frakp}(E_1)$ and $a_{\frakp}(E_2)$ differ.
However, Serre has subsequently remarked \cite[p. 715, note 632.6]{serre-oeuvres} 
that by replacing the mod-$\ell$ argument with an $\ell$-adic argument,
one can improve these bounds to $O((\log N)^2)$; since the details have not appeared in
print elsewhere, Serre has kindly provided them and permitted us to reproduce them
as Theorem~\ref{T:Serre bound} and Corollary~\ref{C:Serre bound} below.
This suggests the possibility of a similar improvement to Theorem~\ref{T:distinguish elliptic curves};
see Remark~\ref{R:Serre remark}.
\end{remark}

\begin{theorem}[Serre] \label{T:Serre bound}
Let $\Gamma$ be a group, let $\ell$ be a prime number, let $r$ be a positive integer, and let
$\rho_1, \rho_2: \Gamma \to \GL_r(\ZZ_\ell)$ be two homomorphisms with distinct traces.
Then there exist a finite quotient $G$ of $\Gamma$ and a nonempty subset $C$ of $G$ with the following properties.
\begin{enumerate}
\item[(a)]
The order of $G$ is at most $\ell^{2r^2} - 1$.
\item[(b)]
For any $\gamma \in \Gamma$ whose image in $G$ belongs to $C$, $\Trace(\rho_1(\gamma)) \neq \Trace(\rho_2(\gamma))$.
\end{enumerate}
\end{theorem}
\begin{proof}
The argument is based on the proof of \cite[Satz~6]{faltings}.
Let $M$ be the (noncommutative) ring of $r \times r$ matrices over $\ZZ_\ell$, and let $A$ be the $\ZZ_\ell$-subalgebra of $M \times M$ generated by the image of $\rho_1 \times \rho_2: \Gamma \to \GL_r(\ZZ_\ell) \times \GL_r(\ZZ_\ell)$. Let $G$ be the image of $\Gamma$ in $A/\ell A$; since $\rank(A) \leq 2r^2$, 
the order of $G$ is at most $\ell^{2r^2}-1$. To define $C$, let $m$ be the largest nonnegative integer such that
$\Trace(\rho_1(\gamma)) \equiv \Trace(\rho_2(\gamma)) \pmod{\ell^m}$ for all $\gamma \in \Gamma$; this integer
exists because we assumed that $\rho_1, \rho_2$ have distinct traces. Define the linear form
$\lambda: A \to \ZZ_\ell$ by
\[
\lambda(x_1, x_2) = \ell^{-m} (\Trace(x_1) - \Trace(x_2));
\]
by reduction modulo $\ell$, $\lambda$ defines a linear form $\overline{\lambda}: A/\ell A \to \FF_\ell$.
Let $C$ be the set of $g \in G$ for which $\overline{\lambda}(g) \neq 0$; this set is nonempty by the choice  of $m$.
For any $\gamma \in \Gamma$ whose image in $G$ belongs to $C$, $\Trace(\rho_1(\gamma)) \not\equiv \Trace(\rho_2(\gamma)) \pmod{\ell^{m+1}}$.
\end{proof}

\begin{cor} \label{C:Serre bound}
Assume the Riemann hypothesis for Artin $L$-functions.
Then there exists a prime number $p$ not dividing $N$
with $p = O((\log N)^2)$ such that $a_{p}(E_1) \neq a_{p}(E_2)$.
\end{cor}
\begin{proof}
Put $\ell = 2$, $r=2$, and $\Gamma = G_{\QQ}$, and apply Theorem~\ref{T:Serre bound} to the $\ell$-adic representations associated to $E_1, E_2$. The resulting group $G$ may be viewed as the Galois group of 
a finite extension of $\QQ$ of absolutely bounded degree unramified away from the primes dividing $N_1 N_2$.
The claim then follows from the effective Chebotarev theorem as stated in \cite[Th\'eor\`eme~6]{serre-chebotarev}.
\end{proof}

\begin{remark} \label{R:Serre remark}
The absence of a factor of $\log \log 2N$ in the bound appearing in Corollary~\ref{C:Serre bound}
is a consequence of \cite[Th\'eor\`eme~5]{serre-chebotarev}, which is a refinement
to effective Chebotarev described at the end of \cite{lagarias-odlyzko}.
Without such a refinement, the bound would have the form $O((\log N)^2 (\log \log 2N)^4)$
as indicated in the remarks following \cite[Th\'eor\`eme~5]{serre-chebotarev}.

This analysis suggests that it should also be possible to obtain a bound of the form
$O((\log N)^2)$ in Theorem~\ref{T:distinguish elliptic curves} by refining the analysis of \cite{kumar}
in the style of \cite[pp.\ 461--462]{lagarias-odlyzko}. We have not attempted to do this.
\end{remark}

\section{Notes on the general case}
\label{sec:general case}

We conclude by returning to the case of a general motive $M$ and sketching how to derive effective
equidistribution under the assumption of Conjecture~\ref{conj:motivic L-functions}. 

Given a function $F: \Conj(G) \to \CC$, one would like to approximate $\sum_{\Norm(\frakp) \leq x,
\frakp \ndiv N} F(\frakp)$ by writing $F$ in terms of characters using \eqref{eq:peter-weyl},
truncating the approximation by discarding the characters of large dimension, then
applying \eqref{eq:explicit formula2} to the characters of small dimension.
To get a meaningful result, it may be necessary to first replace $F$ by a close approximation for which
the coefficients in \eqref{eq:peter-weyl} decay sufficiently rapidly.

In case $G$ is connected, it is not hard to explain how to explicitly carry out these steps in terms
of classical Lie theory. Let $H$ be a Cartan subgroup of $G$, and identify $\Conj(G)$ with the quotient of $H$
by the action of the Weyl group $W$. The $\ZZ$-module of characters of $G$ may then be identified with the
$W$-invariant part of the $\ZZ$-module of characters of $H$. If we fix a Weyl chamber in the lattice of characters of $H$, then each element of the Weyl chamber is the highest weight of a unique irreducible representation of $G$,
whose full character is computed by the Weyl character formula. Any function $F: \Conj(G) \to \CC$
corresponds naturally to a $W$-invariant function $H \to \CC$; the coefficients computed by the Weyl character formula then provide the change-of-basis matrix converting the expansion of $F$ in \eqref{eq:peter-weyl} into the usual Fourier expansion of $F$. But this change-of-basis matrix is triangular, so we can invert it to
convert the Fourier expansion into the expansion in terms of characters. In the case $G = \SU(2)$, 
the Weyl character formula produces \eqref{eq:SU2 characters},
so this construction specializes to the conversion of \eqref{eq:F expansion1} into \eqref{eq:F expansion2}.

Suppose however that the coefficients of $F$ in \eqref{eq:peter-weyl} do not themselves converge sufficiently rapidly to obtain the desired estimates. We then take a small $H$-invariant neighborhood $U$ of the identity in $H$
and define $g: H \to \CC$ to be the characteristic function of $U$ rescaled so that its integral equals 1.
The ordinary Fourier coefficients of $g$ do decay: if we fix a basis $\chi_1,\dots,\chi_n$ of characters of $H$,
where $n = \rank(G)$, then the coefficient of $\chi_1^{m_1} \cdots \chi_n^{m_n}$ is on the order of
$\prod_{i=1}^n (|m_i|+1)^{-1}$. Taking the convolution $F \ast g^r$ over $H$ corresponds to pointwise multiplication
of Fourier coefficients, so we may achieve any desired polynomial decay of Fourier coefficients without changing $F$ too much. For $r$ sufficiently large, this decay persists when we convert Fourier coefficients into character coefficients. In the case $G = \SU(2)$, this process with $F$ taken to be the characteristic function of an interval gives precisely the function $F_{A,B}$ considered by Vinogradov.

In general, the group $G$ need not be connected; for example, for $E$ an elliptic curve over $K$ with
complex multiplication not defined over $K$, $G$ is the normalizer of $\SO(2)$ in $\SU(2)$, which has two connected components. In this case, one can carry out the above analysis after restricting the representations involved to the connected part of $G$.

\section*{Acknowledgements}

This paper arose from discussions with Francesc Fit\'e and Kristin Lauter.
Thanks to them and to Dorian Goldfeld, Jeffrey Hoffstein, Kumar Murty, Jeremy Rouse, Jean-Pierre Serre, Freydoon Shahidi, Igor Shparlinski, and Andrew Sutherland for helpful remarks.

\end{document}